\documentclass[12pt, a4paper]{amsart}
\usepackage{amsmath, amssymb ,amsthm, amsfonts, amsgen}
\usepackage[utf8]{inputenc}
\usepackage[dvipsnames]{xcolor}
\usepackage[colorlinks, allcolors=Fuchsia,pdfstartview=,pdfpagemode=UseNone]{hyperref} 
\usepackage{tikz}
\usepackage{graphicx} 

\numberwithin{equation}{section} 

\definecolor{vg}{rgb}{0.0, 0.26, 0.15}

\newcommand{\Rb}{{\mathbb{R}}}

\newcommand{\R}{{\mathcal{R}}}

\newcommand{\comment}[1]{\vskip.3cm
\fbox{%
\parbox{0.93\linewidth}{\footnotesize #1}}
\vskip.3cm}

\def\essinf{\operatornamewithlimits{ess\,inf}}

\def\supp{\operatornamewithlimits{supp}}

\renewcommand{\div}{\divop}
\renewcommand{\phi}{\phi}
\def\le{\leqslant}
\def\leq{\leqslant}
\def\ge{\geqslant}
\def\geq{\geqslant}
\def\phi{\varphi}
\def\rho{\varrho}
\def\epsilon{\varepsilon}
\def\vartheta{\theta}

\def\supp{\operatorname{supp}}
\def\esssup{\operatornamewithlimits{ess\,sup}}

\def\div{\qopname\relax o{div}}

\def\dist{\qopname\relax o{dist}}
\def\px{{p(\cdot)}}

\def\loc{{\rm loc}}
\def\BV{{\rm BV}}

\newcommand{\inc}[1]{\hyperref[def:aInc]{{\normalfont(Inc){\ensuremath{_{#1}}}}}}
\newcommand{\dec}[1]{\hyperref[def:aDec]{{\normalfont(Dec){\ensuremath{_{#1}}}}}}
\newcommand{\ainc}[1]{\hyperref[def:aInc]{{\normalfont(aInc){\ensuremath{_{#1}}}}}}
\newcommand{\adec}[1]{\hyperref[def:aDec]{{\normalfont(aDec){\ensuremath{_{#1}}}}}}
\newcommand{\adeci}[1]{\hyperref[def:aDeci]{{\normalfont(aDec){\ensuremath{_{#1}^\infty}}}}}
\newcommand{\azero}{\hyperref[def:a0]{{\normalfont(A0)}}}
\newcommand{\aone}{\hyperref[def:a1]{{\normalfont(A1)}}}

\newcommand{\aones}[1]{\hyperref[def:a1s]{{\normalfont(A1-{\ensuremath{{#1}})}}}}

\newcommand{\Phiw}{\Phi_{\text{\rm w}}}
\newcommand{\Phic}{\Phi_{\text{\rm c}}}

\makeatletter
\def\rightharpoonupfill@{\arrowfill@\relbar\relbar\rightharpoonup}
\newcommand{\xrightharpoonup}[2][]{\ext@arrow
0359\rightharpoonupfill@{#1}{#2}} \makeatother

\def\dist{\text{dist}}

\def\e{{\epsilon}}

\def\x{{\times}}

\def\R{{\mathbb  R}}
\def\Rn{{\mathbb  R^n}}

\def\esssup{\operatornamewithlimits{ess\,sup}}
\def\essinf{\operatornamewithlimits{ess\,inf}}

\newcommand{\Hzero}{\hyperref[Hzero]{{\normalfont(H0)}}}
\newcommand{\Hone}{\hyperref[Hone]{{\normalfont(H1)}}}
\newcommand{\Htwo}{\hyperref[Htwo]{{\normalfont(H2)}}}
\newcommand{\Hthree}{\hyperref[Hthree]{{\normalfont(H3)}}}

\newtheorem{theorem}[equation]{Theorem}
\newtheorem{lemma}[equation]{Lemma}
\newtheorem{proposition}[equation]{Proposition}

\theoremstyle{definition}
\newtheorem{definition}[equation]{Definition}

\theoremstyle{remark}
\newtheorem{remark}[equation]{Remark}
\numberwithin{equation}{section}

\newcommand{{\rr}}{{\mathbb R}}

\newcommand{\N}{{\mathbb N}}
\newenvironment{@abssec}[1]{%
 \if@twocolumn
   \section*{#1}%
 \else
   \vspace{.05in}\footnotesize
   \parindent .2in
 {\upshape\bfseries #1. }\ignorespaces
 \fi}
 {\if@twocolumn\else\par\vspace{.1in}\fi}

\begin{document}

\title[Quasi-convex functionals with variable exponent growth]
{Relaxation of quasi-convex functionals with variable exponent growth}
	
	
\author[G.\ Bertazzoni]{Giacomo Bertazzoni}
\address{G.\ Bertazzoni, Dipartimento di Scienze Fisiche, Informatiche e Matematiche, Università degli Studi di Modena e Reggio Emilia, via Campi 213/B, 41125 Modena, Italy}
\email{\href{mailto:giacomo.bertazzoni@unimore.it}{\tt giacomo.bertazzoni@unimore.it}}

\author[P.\ Harjulehto]{Petteri Harjulehto}
\address{P.\ Harjulehto,
Department of Mathematics and Statistics,
FI-00014 University of Helsinki, Finland}
\email{\href{mailto:petteri.harjulehto@helsinki.fi}{\tt petteri.harjulehto@helsinki.fi}}

\author[P.\ Hästö]{Peter Hästö}
\address{P.\ Hästö, Department of Mathematics and Statistics,
FI-00014 University of Helsinki, Finland}
\email{\href{mailto:peter.hasto@helsinki.fi}{\tt peter.hasto@helsinki.fi}}

\author[E.\ Zappale]{Elvira Zappale}
\address{E.\ Zappale, Dipartimento di Scienze di Base e Applicate per l’Ingegneria, Sapienza - Universit\`{a} di Roma,
			via Antonio Scarpa, 16, 00161 Roma, Italy} 
\email{\href{mailto:elvira.zappale@uniroma1.it}{\tt elvira.zappale@uniroma1.it}}

\date{\today}

\begin{abstract}
We prove a relaxation result for a quasi-convex bulk integral functional 
with variable exponent growth in a suitable space of bounded variation type.
A key tool is a decomposition under mild assumptions of the energy into 
absolutely continuous and singular parts weighted via a recession function.
\end{abstract}


\maketitle


\section{Introduction and main result}\label{sect:introduction}

Variable exponent spaces have been studied since the 1980s \cite{KR, S}. 
In the 1990s, motivated by the study of composite materials, Zhikov focused on variational integrals with non-standard growth \cite{Z1, Z2, Z3, Z5, ZKO}.
Since then integral functionals defined in variable exponent spaces have been actively studied, both in the context of 
regularity theory and applications to electrorheological fluids and homogenization, see \cite{DieHHR11} and its bibliography.

Existence theory of several models have been considered, related to optimal design and double phase problems, thin structures, dielectric materials, homogenization and fluids: we mention 
\cite{BZ1, BZ2, BEZ, BHH, CM, EP, ElePZ24, MM, Sy}, among a wider literature, where lower semicontinuity, relaxation and variational convergence for integral functionals with densities satisfying a
$\px$-growth condition have been considered, under a variety of continuity assumptions on the exponent $p$.
In these papers integral representations are obtained, more or less explicitly, according to the regularity hypotheses on $p$.

In the Sobolev space $W^{1,p(\cdot)}(\Omega; \mathbb R^m)$ in \cite{CM, ElePZ24, Sy} the 
$\Gamma$-limit (see \cite{Dal93}) of such energy functionals is still an integral functional of the same type and
growth, provided the exponent is $\log$-H\"older continuous (Definition~\ref{logH}). This condition ensures that we can freeze the exponent on small balls and use blow-up methods introduced by 
Fonseca and Müller \cite{FonM92, FonM93}. Similarly, lack of $\log$-H\"older continuity entails a singular term 
in the measure representation of functionals, see \cite{Z5} and also \cite{BZ1,BZ2}.

An additional challenge arises in variational problems related to imaging problems where 
linear growth ($p=1$) energies are useful.
Chen, Levine and Rao \cite{CLR} proposed a variable exponent generalization of the so-called 
ROF total variation model in classical BV-space. This was later extended to variable 
exponent \cite{HarHL08, HarHLT13, LLP10} and double phase \cite{HarH21} BV-spaces, and recently to general Orlicz $\BV^\phi$-spaces \cite{EleHH25}. However, the models which combine linear and 
super-linear growth have so far only been considered for convex energies while 
for linear growth energies there is a well-developed theory of quasi-convex energies, see 
\cite[Section~5.5]{AmbFP00}.

More recently, another point of view has been considered in \cite{ARS, SSS}, taking into account 
not only bulk energies but their coupling with surface energies, thus leading to the study 
of variable exponent versions of so-called Mumford--Shah-type functionals 
(see \cite[Chapter~6]{AmbFP00}) with bulk energy of type
\begin{equation}\label{intintro}
\int_\Omega f(\nabla u)^{p(x)} dx,
\end{equation} 
for superlinear growth, $\inf_\Omega p>1$. 
In these papers' framework of special functions with bounded variation and variable exponent growth, quasiconvexity of the bulk integrand appears naturally as necessary and sufficient for a lower semicontinuous energy, as in the standard growth case.

When dealing with composite materials, we can describe failure phenomena, such
as fracture, by zones of the domain with linear energy behavior in the relaxation process.
Orlicz--Sobolev space bulk energies cannot captured this, which leads us to the extension 
of the above non-standard growth functional \eqref{intintro} to a free discontinuity
setting, where singularities may appear in the form of jump discontinuities or Cantor measures. 

The main contribution of this paper is to introduce the first model which features both 
quasi-convex anisotropy of \cite{ADM, FonM92, FonM93} and the combination of linear and super-linear growth of \cite{EleHH25, HarHL08, HarHLT13}. 
We provide sufficient conditions for lower semicontinuity and relaxation of
functionals of the form \eqref{intintro} when the exponent $p$ may take the value $1$. Once again, the $\log$-H\"older continuity ensures the lack of the Lavrentiev phenomenon and 
that the only concentration effect is the expected one from the BV-space described by the 
recession function.
However, the subtlety of the problem is illustrated by the fact that a slightly stronger 
vanishing $\log$-Hölder continuity is required in the linear growth set $\{p=1\}$.

To state our main result let $A\subset \Omega\subset \mathbb R^n$ be open. 
We consider the lower semicontinuous envelope of \eqref{intintro} with respect to the strong $L^{p(\cdot)}$ convergence, i.e.\ the functional
\begin{equation}\label{calF}
\mathcal F(u,A)
:=
\inf\bigg\{\liminf_{h\to \infty}\int_A f(\nabla u_h)^{p(x)}\, dx \,\Big|\, u_h \in W^{1,\px}(A;\R^m), u_h \to u \text{ in } L^{\px}(A;\R^m)\bigg\}
\end{equation}
for $u\in L^\px(\Omega; \R^m)$; we abbreviate $\mathcal F(u):=\mathcal F(u, \Omega)$. 
We obtain an integral representation in the space of functions with 
generalized growth bounded variation of this functional in terms of the point-wise (weak) 
\emph{recession function} 
\[
f^\infty(\xi):=\limsup_{t\to\infty}\frac{f(t\xi)}{t}.
\]

\begin{theorem}\label{thm:main}
Let $\Omega$ be a bounded open set with Lipschitz boundary,
$p: \Omega\to [1,\infty)$ be $\log$-Hölder continuous and
assume that $f:\mathbb R^{m\times n}\to [0,\infty)$ satisfies the following assumptions. 
\begin{itemize}
\item[\Hzero{}] $f$ is continuous and $f(0)=0$. \label{Hzero}
\item[\Hone{}]\label{Hone}
$\displaystyle f(\xi)\leq\int_{(0, 1)^n} f(\xi + \nabla u)\,dx$
for every $\xi\in\mathbb{R}^{n\times m}$ and $u \in W^{1,\infty}_0((0, 1)^n;\mathbb R^m)$
(quasiconvexity).
\item[\Htwo{}] There exists $m>0$ such that $m\,|\xi| \le f(\xi)$ for every $\xi \in \mathbb R^{n\times m}$. \label{Htwo}
\item[\Hthree{}] There exists $M\ge 1$ such that  $f(\xi)\le M(1+|\xi|)$ for every $\xi \in \mathbb R^{n\times m}$. \label{Hthree}
\end{itemize}
Then the relaxed functional has the integral representation
\[
\mathcal F(u,\Omega)=\int_\Omega f(\nabla u)^{p(x)} \,dx 
+ \int_{\{x\in \Omega\,|\, p(x)=1\}} f^\infty\Big(\frac{d D^s u}{d |D^s u|}\Big) \,d|D^s u|
\]
for all $u \in \BV^\px(\Omega; \R^m)$.
\end{theorem}

\begin{remark}\label{remfunct=}
If $A$  has Lipschitz boundary and if $f$ satisfies \Hthree{}, 
then 
\[
\mathcal F(u,A)
=
\inf\bigg\{\liminf_{h\to \infty}\int_A f(\nabla u_h)^{p(x)}\, dx \,\Big|\, u_h \in W^{1,p(\cdot)}(A;\R^m), u_h \to u \text{ in } L^1(A;\R^m)\bigg\}
\]
with convergence in $L^1$ instead of $L^\px$, by \cite[Proposition 2.8]{ElePZ24}.
\end{remark}

Our proof is based on approximation by linear growth functionals, for which we can use 
an improved version of the linear result due to Breit, Diening and Gmeinder \cite{BreDG20}, see Lemma~\ref{lem:lsc}. 
Based on this lemma we obtain in Proposition~\ref{prop:Fqc} a lower semi-continuity result 
which allows us to prove the inequality ``$\ge$'' in the main theorem. 
The proof of the opposite inequality uses the so-called blow-up method from 
\cite[Section~5.5]{AmbFP00}, which, as a first step, requires us to prove that $\mathcal F(u,\cdot)$ in \eqref{calF} is a Radon measure absolutely continuous with respect to the generalized total variation of $Du$, in the sense of \cite{EleHH25,HarHL08, HarHLT13}. 
In Proposition~\ref{prop:Fcalmeas} we refine techniques from \cite{HarHLT13} 
to show how the strong $\log$-Hölder continuity allows us 
to separate the effect of the exponent into a multiplicative constant which tends to $1$ in 
the blow-up process. The remaining linear growth, quasi-convex part is then handled as in \cite[Section~5.5]{AmbFP00}. Before the main result, we introduce background material 
in Section~\ref{sect:preliminaries} and the appropriate vector-valued BV-type space with variable growth 
in Section~\ref{sect:BV}.


\section{Preliminaries}\label{sect:preliminaries}

We always consider a bounded open set $\Omega$ in $\Rn$, $n \ge 2$. 
For a set $E\subset \Rn$, $\chi_E$ is the \textit{characteristic function} of $E$ such that $\chi_E(x)=1$ if $x\in E$ and $\chi_E(x)=0$ if $x\not\in E$. We denote the \emph{H\"older conjugate} exponent of $p \in [1,\infty]$ by $p'=\frac{p}{p-1}$. A generic constant denoted by $c>0$ without subscript may change between appearances.

Let $f, g : E \to \R$. The notation $f\lesssim g$ means that
there exists $C>0$ such that $f(y)\le Cg(y)$ all $y\in E$ and 
$f\approx g$ means that $f\lesssim g\lesssim f$.
When $E\subset \R$, we say that $f$ is \textit{almost increasing} on $E$ with constant 
$L\ge 1$ if $f(s)\le L f(t)$ whenever $s,t \in E$ with $s\le t$. If we can choose $L=1$, 
we say that $f$ is \textit{increasing} on $E$. \textit{Almost decreasing} and \textit{decreasing} are defined similarly. 

\subsection*{Generalized Orlicz spaces}

We introduce a general framework for function spaces following \cite{HarH19}. 
We say  that $\phi: \Omega\times [0, \infty] \to [0, \infty]$ is a 
\textit{weak $\Phi$-function}, and write $\phi \in \Phiw(\Omega)$, if 
the following conditions hold for every $x \in \Omega$:
\begin{itemize}
\item 
$\phi(\cdot, |f|)$ is measurable for every measurable function $f:\Omega\to \overline\R$.
\item
$t \mapsto \phi(x, t)$ is increasing. 
\item 
$\displaystyle \phi(x, 0) = \lim_{t \to 0^+} \phi(x,t) =0$ and $\displaystyle \phi(x,\infty)= \lim_{t \to \infty}\phi(x,t)=\infty$.
\item 
$t \mapsto \frac{\phi(x, t)}t$ is $L$-almost increasing on $(0,\infty)$ with 
constant $L$ independent of $x$.
\end{itemize}
A weak $\Phi$-function $\phi$ is a \textit{convex $\Phi$-function} if it is additionally convex and left-continuous with respect to second variable, denoted $\phi \in \Phic(\Omega)$.
For instance, $t \mapsto \frac{1}{p} t^{p}$, $t \mapsto  t^{p(x)}$ and $t \mapsto \frac{1}{p(x)} t^{p(x)}$ are  convex  $\Phi$-functions, where $p: \Omega \to [1, \infty)$ is measurable, whereas $t\mapsto \min\{t, t^2\}$ is a weak $\Phi$-function which is not convex.

Given $\phi \in \Phiw(\Omega)$, we define the \textit{quasimodular} of 
the measurable function $u:\Omega \to [- \infty, \infty]$ as
\[
\rho_\phi(u):=\int_\Omega \phi(x, |u(x)|)\,dx,
\]
and we write $u\in L^\phi(\Omega)$ if there exists $\lambda >0$ such that
$\rho_\phi(\lambda u)< \infty$.
The space $L^\phi(\Omega)$ is endowed with the quasinorm
\[
\|u\|_{L^\phi(\Omega)}
:=
\|u\|_{\phi}:=\inf\{\lambda >0 \,|\,  \rho_{\phi}(\tfrac u\lambda)\le 1\}.
\]
More information about generalized Orlicz spaces can be found from \cite{HarH19}.

The extension to the vector valued setting is the following. 
Let $\phi \in\Phiw(\Omega)$. 
A measurable function $u\equiv (u_1,\dots, u_m) \colon \Omega \to \R^m$ belongs to  $L^\phi(\Omega;\R^m)$  if
$u_\alpha \in L^{\phi}(\Omega)$ for every $\alpha = 1, \ldots, m$. 
For $u \colon \Omega \to \R^m$ we write
\begin{equation*}
\|u\|_{\phi}:= \big\| |u| \big\|_{\phi} < \infty,
\end{equation*}
and note that
\[
L^{\phi}(\Omega;\R^m) 
= 
\big\{u:\Omega \to \R^m \hbox{ measurable } \,\big|\, \rho_\phi(|\lambda  u|) < \infty \quad \textnormal{for some }\lambda > 0\big\}.
\] 

Next we define some standard conditions on the $\Phi$-functions that imply 
additional properties of the spaces.

\begin{definition}\label{def:phi-conditions}
Let $\phi : \Omega \times [0,\infty] \to [0,\infty]$ and let $p,q>0$. Then we define the following conditions:
\begin{itemize}
\item[(A0)]\label{def:a0}
There exists $\beta \in (0,1]$ such that $\phi(x,\beta) \le 1 \le \phi(x, \frac{1}{\beta})$ for every $x \in \Omega$.
\item[(A1)]\label{def:a1} 
For every $K>0$ there exists $\beta \in (0,1]$ such that, for every $x,y\in \Omega$,
\[
\phi(x,\beta t) \le \phi(y,t)+1 \quad \textnormal{when } \ \phi(y, t) \in \Big[0, \frac{K}{|x-y|^n}\Big].
\]
\item[(aInc)$_p$]\label{def:aInc}
There exists $L_p \ge 1$ such that $t \mapsto \frac{\phi(x,t)}{t^p}$ is $L_p$-almost 
increasing on $(0,\infty)$ for every $x \in \Omega$.
\item[(aDec)$_q$]\label{def:aDec} 
There exists $L_q \ge 1$ such that $t \mapsto \frac{\phi(x,t)}{t^q}$ is $L_q$-almost decreasing on $(0,\infty)$ for every $x \in \Omega$.
\end{itemize}
We say that \ainc{} or \adec{} holds if \ainc{}$_p$ or \adec{}$_q$ holds for some $p>1$ or $q<\infty$, respectively.
\end{definition}

Note that this slight variation of \aone{} is equivalent to that of \cite{HarH19} if \azero{} and \adec{} 
hold (see \cite[Section~3]{EleHH25}). 
Assumption \aone{} is an almost continuity condition whereas \ainc{} and 
\adec{} are quantitative versions of the $\nabla_2$ and $\Delta_2$ conditions 
from Orlicz space theory and measure lower and upper growth rates. 

We denote by $\phi^*$ the conjugate function of $\phi$, i.e. 
\[
\phi^*(x, t):= \sup_{s \ge 0}\{ st - \phi (x, s)\},
\] 
for $x \in \Omega$ and $t \in [0,\infty)$.
We define the associate space of the generalized Orlicz space $L^\phi(\Omega)$, 
a variant of the dual function space which 
works better at the end-points $p=1$ and $p=\infty$. 

\begin{definition}
\label{def:dual}
Let $\phi \in \Phiw(\Omega)$. The \emph{associate space $(L^\phi(\Omega))'$} 
is the subset of measurable functions with finite norm
\begin{equation*}
\|f\|_{(L^\phi(\Omega))'} := \sup_{\|g\|_\phi \le 1} \int_\Omega fg \,dx.
\end{equation*}
\end{definition}

For weak $\Phi$-functions the associate space 
equals the generalized Orlicz space $L^{\phi^*}(\Omega)$ \cite[Theorem 3.4.6]{HarH19}, 
up to equivalence of norms.

\subsection*{Variable exponent spaces}

A measurable function $p :\Omega \to [1,\infty)$ is called a \emph{variable exponent}.
We denote for $A \subset \Omega$,
\begin{align*}
p^+_A:= \esssup_{x \in A}p(x), \; p^-_A:= \essinf_{x \in A} p(x), \;  p^+:= p^+_\Omega, \; 
p^-:=p^-_\Omega.
\end{align*}
By the symbol $Y$  we denote
the set where $p$ equals one, 
\[
Y := \{x \in \Omega \mid p(x) = 1\}.
\]

Let $u: \Omega \to [-\infty, \infty]$. We recall that the variable exponent modular is defined as 
\begin{align*}
\rho_{L^\px(\Omega)}(u) := \rho_\px(u) := \int_\Omega \tfrac1{p(x)}|u(x)|^{p(x)}\,dx.
\end{align*}
This defines the \textit{variable exponent Lebesgue space $L^\px(\Omega;\R^m)$} as 
the special case of $L^\phi(\Omega;\R^m)$ with $\phi(x,t)=\frac1{p(x)}t^{p(x)}$. 
Consider Definition~\ref{def:phi-conditions} in this context. 
Now \azero{} and \ainc{p^-} always hold and \adec{p^+} holds if $p^+<\infty$
\cite[Lemma~7.1.1]{HarH19}. 
For \aone{}, we need the $\log$-Hölder continuity of the exponent \cite[Proposition~7.1.2]{HarH19}.

In the variable exponent case, the conjugate $\Phi$-function is related to the point-wise Hölder conjugate exponent \cite[Example 4.3]{EleHH25}: 
\begin{equation*} 
\phi^*(x,s)= \begin{cases}
\frac{1}{p'(x)}s^{p'(x)}, &\hbox{ if } x \in \Omega \setminus Y, \\
0, &\hbox{ if }x \in Y \text{ and }s\le 1, \\
\infty, &\hbox{ if }x \in Y \text{ and }s> 1. \\
\end{cases} 
\end{equation*}
We will abbreviate this function as $\frac{1}{p'(x)}s^{p'(x)}$ 
also in $Y$.

\begin{definition}\label{logH}
A variable exponent $p$ is \emph{$\log$-H\"older continuous} if there exists $C >0$ such that
\[
|p(x)-p(y)|\le \frac C {\log(e+ \frac{1}{|x - y|})}
\]
for all $x,y \in \Omega$.
It is \emph{strongly $\log$-H\"older continuous} if additionally
\[
p(x)-1 \le \frac{\omega(|x-y|)}{\log(e+ \frac{1}{|x - y|})},
\]
for some $\omega:[0,\infty)\to [0, \infty)$ with $\lim_{t \to 0}\omega(t)=0$ and every $y \in Y$ and $x\in\Omega$.
\end{definition}

A bounded exponent $p$ is $\log$-H\"older continuous in $\Omega$ if and only if there exists a constant $C >0$
such that
\[
\mathcal L^n(B)^{p^-_B-p^+_B} \le C
\]
for every ball $B \subset \Omega$ \cite[Lemma 4.1.6]{DieHHR11}. 
Under the $\log$-H\"older condition, smooth functions are dense in variable exponent
Sobolev spaces \cite[Theorem 9.1.8]{DieHHR11}. Strong $\log$-Hölder continuity 
\cite{HarHL08} is a generalization of vanishing $\log$-Hölder continuity, 
where the $\log$-Hölder constant is controlled in only 
the set $Y$ rather than in all of $\Omega$.

We will need the following technical lemma. Note that we do not  
need the $\log$-Hölder continuity in all of $\Omega$, only the stronger version in 
the set $Y$, so we say that a variable exponent $p$
\emph{ is strongly $\log$-H\"older continuous in $Y$} if
\[
p(x)-1 \le \frac{\omega(|x-y|)}{\log(e+ \frac{1}{|x - y|})},
\]
for some $\omega:[0,\infty)\to [0, \infty)$ with $\lim_{t \to 0}\omega(t)=0$ and every $y \in Y$ and $x\in\Omega$.

\begin{lemma}\label{lem:smallOnY}
Suppose that $p$ is strongly $\log$-Hölder continuous in $Y$. 
If $\rho_{p'(\cdot)}(w)<\infty$ for some continuous $w$, then 
$|w|\le 1$ in $Y$
\end{lemma}
\begin{proof}
Suppose that $|w(y_0)|>1$ for some $y_0\in Y$. Let $B\subset \Omega$ be 
a ball centered at $y_0$ such that $|w|>c_0>1$ in $B$. By the strong $\log$-Hölder continuity 
condition in $Y$, 
\[
p'(x) \ge \frac1{p(x)-1} \ge M \log(e+\frac1{|x-y_0|}),
\]
where $M$ can be made large by restricting $x$ to a small ball centered at $y$. 
Consequently,
\[
\int_\Omega |w|^{p'(x)}\, dx
\ge 
\int_B c_0^{M \log(e+\frac1{|x-y_0|})}\, dx
\ge
\int_B |x-y_0|^{-M \log c_0}\, dx,
\]
where in view of \cite[Lemma 3.1.6]{DieHHR11} the left hand side is equivalent to $\rho_{p'(\cdot)}(w)$. 
In a small ball centered at $y_0$, $M\log c_0 \ge n$ and the integral diverges.
Thus, if the integral converges, then $|w|\le 1$, as claimed.
\end{proof}

The \emph{variable exponent Sobolev space $W^{1,p(\cdot)}(\Omega)$} consists of functions $u \in L^{p(\cdot)}(\Omega)$ whose distributional gradient
$ Du$ belongs to $L^{p(\cdot)}(\Omega)$. The variable exponent Sobolev space $W^{1,p(\cdot)}(\Omega)$ is a Banach space with the norm
\[
\|u\|_{W^{1, \px}(\Omega)}:= \|u\|_{1, \px}:= \|u\|_{L^\px(\Omega)}+ \| D u\|_{L^\px(\Omega)}
\]
The norm of the latter term is the $L^{p(\cdot)}$-norm of the 
scalar-valued function $|Du|$, the Euclidean norm of $Du$, 
i.e.\ $\|Du\|_{L^\px(\Omega)} := \big\| |Du| \big\|_{L^\px(\Omega)}$.

\subsection*{Functions of linear growth}

We collect here some results about functions with linear growth that will be used in the 
proofs. 

The following is a special case of Proposition~5.1, \cite{BreDG20}, with $\mathbb A = D$. 
This special case was originally proved by Fonseca and M\"uller \cite{FonM93} with an additional 
assumption (H5), which we avoid by using the more general recent reference. 

\begin{lemma}[Weak lower semicontinuity]
\label{lem:lsc}
Let $\Omega$ be a bounded open set with Lipschitz boundary.
Assume that $f:\overline \Omega\times \R^m\to \R$ is continuous, 
$f(x, \cdot)$ satisfies \Hone{}, \Htwo{} and \Hthree{} uniformly in $x\in \overline\Omega$ 
and 
\[
|f(x,\xi)-f(y,\xi)| \le \omega(|x-y|)(1+|\xi|)
\] 
for some modulus of continuity $\omega$. 
Then $F_{qc}^1 : \BV(\Omega;\R^m) \to [0, \infty)$, defined as 
\[
F_{qc}^1(u)
:=
\int_\Omega f(x, \nabla u)\,dx + \int_\Omega  f^\infty_s\Big(x, \frac{d D^s u}{d |D^s u|}\Big)\, d |D^s u|,
\]
is sequentially weakly* lower semicontinuous with respect to $\BV(\Omega;\R^m)$-convergence, 
where $f^\infty_s$ is the \emph{strong recession function}:
\[
f^\infty_s (x, \xi) := 
\lim_{\substack{t\to \infty\\\xi'\to \xi\\ x' \to x}} \frac{f(x', t\xi')}{t}.
\]
\end{lemma}

\begin{lemma}[Reshetnyak continuity, Theorem~2.38, \cite{AmbFP00}]\label{lem:Reshetnyak}
Let $\Omega\subset \R^n$ be open and $\mu$ and $\mu_h$ be $\R^m$-valued finite Radon measures in 
$\Omega$. If $\mu_h \to \mu$ weakly* in $\Omega$ and $|\mu_h|(\Omega)\to |\mu|(\Omega)$, 
then
\[
\lim_{h\to\infty}\int_\Omega f\Big(x, \frac{\mu_h}{|\mu_h|}\Big)\, d|\mu_h|
=
\int_\Omega f\Big(x, \frac{\mu}{|\mu|}\Big)\, d|\mu|
\]
provided $f:\Omega\times S^{m-1}\to \R$ is continuous and bounded, 
where $S^{m-1}\subset \R^m$ is the unit sphere.
\end{lemma}

\section{BV-type spaces}\label{sect:BV}

Following \cite{AmbFP00}, we say that a function $u \in L^1(\Omega)$ has a \emph{bounded variation}, 
denoted $u \in \BV(\Omega)$, if 
\[
V(u,\Omega):= \sup\bigg\{\int_\Omega u  \div w \,dx \,\Big|\, w \in C^1_0(\Omega;\R^m), |w|\le 1\bigg\}< \infty.
\]
Such functions have distributional first derivatives $Du$ which are Radon measures 
\cite[Section 3.1]{AmbFP00}
and $V(u,\Omega)$ equals the total variation $|Du|(\Omega)$ of the measure $Du$ \cite[Proposition 3.6]{AmbFP00}. 

In the sequel, we
use the Lebesgue decomposition
\[Du = D^au + D^su,
\]
where $D^au$ is the absolutely continuous part of the derivative and $D^su$ is the singular part 
(with respect to the Lebesgue measure). The
density of $D^a u$ is the vector valued function $\nabla u$ such that
\[
\int_\Omega w \cdot  d D^a u= \int_\Omega w \cdot \nabla u \,dx,
\]
for all $w \in  C^\infty_0 (\Omega;\mathbb R^n)$. In the sequel we will identify $D^a u$ with $\nabla u$.

We generalize to the vector-valued case a formulation for BV-type spaces introduced in \cite{EleHH25}.
When $m=1$, these spaces reduce to $BV^\phi(\Omega):=BV^\phi(\Omega; \R)$ 
considered in \cite{EleHH25} and when $\phi(t)=t$ they reduce to ordinary BV-spaces.

%
%

\begin{definition}\label{defn:V-rho_m}
Let $m \in \mathbb N$, $\phi \in \Phiw(\Omega)$ and $u \in L^1_\loc(\Omega;\R^m)$. 
We define the ``dual norm'' 
\begin{align*}
V_\phi^m(u,\Omega):= V^m_\phi(u):= 
\sup\Bigg\{\sum_{\alpha=1}^m\int_\Omega u_\alpha \div w_\alpha \,dx \,\Big|\, w \in [C^1_0(\Omega;\R^n)]^{m}, \|w\|_{\phi^*}\le 1\Bigg\},
\end{align*}
and the ``dual modular''
\begin{align*}
\rho_{V,\phi}^m(u):=\sup\Bigg\{ \int_\Omega \bigg(\sum_{\alpha=1}^m u_\alpha \div w_\alpha - \phi^*(x, |w|)\,\bigg) \,dx \,\Big|\, w \in [C^1_0(\Omega; \R^n)]^m\Bigg\}.
\end{align*}
We say that $u\in L^\phi(\Omega;\R^m)$ belongs to $\BV^\phi(\Omega;\R^m)$ if
\[
\|u\|_{\BV^{\phi}(\Omega;\R^m)}:= \|u\|_{L^\phi(\Omega;\R^m)} + V_\phi^m(u, \Omega) < \infty.
\]
If $\phi(x, t) = \frac1{p(x)} t^{p(x)}$, then we replace $\phi$ by $\px$ in the  notation, for example $V_{\px}^m(u,\Omega)$.
\end{definition}

\begin{lemma}\label{BVvectorial}
If $\phi \in \Phiw(\Omega)$, then 
$u \equiv(u_1,\dots, u_m)\in \BV^\phi(\Omega;\R^m)$ if and only if $u_i \in \BV^\phi(\Omega)$, for every $i \in \{1,\dots, m\}$.
\end{lemma}
\begin{proof}
By \cite[Theorem 2.5.10 and Corollary 3.2.5]{HarH19}, we first observe regarding the Lebesgue 
norms that $\|u\|_{\phi} \lesssim \sum_{i=1}^m \|u_i\|_{\phi}$, while $\|u_i\|_{\phi} \lesssim \|u\|_{\phi}$, for every $i=1,\dots, m$. 
Thus it remains to compare the variations. 

Assume first that $u\equiv(u_1,\dots, u_m) \in \BV^\phi(\Omega;\R^m)$.
Let $w_i \in C^1_0(\Omega;\mathbb R^n)$ with $\|w_i\|_{\phi*}\le 1$. 
Define $w :=(0,\dots, 0, w_i, 0, \ldots, 0)$ and note that $w \in [C^1_0(\Omega;\mathbb R^n)]^m$ 
with $\|w\|_{\phi^*}= \|w_i\|_{\phi^*}\le 1$. Then 
\[
\int_\Omega u_i \div w_i \,dx = \int_\Omega \sum_{k=1}^m u_k \div w_k \,dx \le V^m_\phi(u,\Omega).
\]
Taking the supremum over such $w_i$, we find that 
\[
V_\phi(u_i,\Omega)\le V^m_\phi(u,\Omega),
\]
for every $i\in \{1,\dots, m\}$, which proves that $u_i\in \BV^\phi(\Omega)$.

Assume then that $u_i \in \BV^\phi(\Omega)$ for every $i\in \{1,\dots, m\}$ and choose 
$w \in [C^1_0(\Omega;\mathbb R^n)]^m$ with $\|w\|_{\phi*}\le 1$. 
Then each component satisfies $w_i \in C^1_0(\Omega;\mathbb R^n)$ and 
$\|w_i\|_{\phi*}\le 1$ and it follows from the definitions of $V_\phi$ that
\[
\int_\Omega \sum_{i=1}^m u_i \div w_i \,dx \le \sum_{i=1}^m V_\phi(u_i).
\]
Taking the supremum over such $w$, we find that 
\[
V^m_\phi(u,\Omega) \le \sum_{i=1}^m V_\phi(u_i,\Omega),
\]
which proves that $u\equiv(u_1,\dots, u_m) \in \BV^\phi(\Omega;\R^m)$.
\end{proof}

The following lemmas are vector-valued versions of results from \cite{EleHH25}. 
The proofs of Lemmas~\ref{asLemma4.6EleHH25} and 
\ref{lem:density0} are essentially identical and are thus omitted altogether (cf.\ \cite[Lemmas~4.6 and 5.4]{EleHH25}). For the others we show the changed parts.

\begin{lemma}\label{asLemma4.6EleHH25}
If $\phi \in \Phiw(\Omega)$, then $V_\phi^m$ is a seminorm 
and $\|\cdot\|_{\BV^\phi(\Omega;\R^m)}$ is a quasinorm in $\BV^\phi(\Omega;\R^m)$. 
Moreover, if $\phi \in \Phic(\Omega)$, then $\|\cdot\|_{\BV^\phi(\Omega;\R^m)}$ is a norm.
\end{lemma}

The next result shows that $\rho^m_{V,\phi}$ is a 
left-continuous semimodular in $L^1(\Omega;\R^m)$. 
The proof follows \cite[Lemma~4.7]{EleHH25}. 

\begin{lemma}\label{lem:pseudo-modular}
If $\phi\in\Phiw(\Omega)$, then 
\begin{enumerate}
\item $\rho^m_{V,\phi}(0)=0$;
\item the function $\lambda\mapsto\rho^m_{V,\phi}(\lambda u)$ is increasing on $[0,\infty)$ for every $u\in L^1(\Omega;\R^m)$;
\item $\rho^m_{V,\phi} (-u)=\rho^m_{V,\phi}(u)$ for every $u\in L^1(\Omega;\R^m)$;
\item  $\rho^m_{V,\phi}(\theta u +(1-\theta)v ) \leqslant \theta \rho^m_{V,\phi}(u) + 
(1-\theta)\rho^m_{V,\phi}(v)$ 
for every $u,v\in L^1(\Omega;\R^m)$ and every $\theta \in [0,1] $;
\item $\lim_{\lambda \to 1^-} \rho^m_{V,\phi}(\lambda u) = \rho^m_{V,\phi}(u)$ for every $u\in L^1(\Omega;\R^m)$.
\end{enumerate}
\end{lemma}
\begin{proof}
The proofs of properties (1) and (3) coincide verbatim with \cite[Lemma~4.7]{EleHH25}, 
so we omit them.

To show that $\lambda \mapsto \rho^m_{V, \phi}(\lambda u)$ is increasing for every $u$ we
let $\lambda \in (0, 1)$ and $w \in [C^1_0(\Omega; \mathbb R^{n})]^m$.
Since $\phi^*$ is increasing, 
\[
\int_\Omega \Big(\sum_{i=1}^m\lambda u_i \div w_i - \phi^*(x, |w|)\Big)\,dx 
\le \int_\Omega  \Big( \sum_{i=1}^mu_i\div(\lambda w_i) - \phi^*(x, |\lambda w|)\Big)\,dx 
\le \rho^m_{V,\phi}(u),
\]
as $\lambda w \in [C^1_0(\Omega; \mathbb R^{n})]^m$. Taking the 
supremum over $w$, we get $\rho^m_{V,\phi}(\lambda u) \le \rho^m_{V,\phi}(u)$.

To prove convexity of $\rho^m_{V,\phi}$ we let $u, v \in L^1(\Omega;\R^m)$, $\theta \in (0, 1)$ and $w\in [C^1_0(\Omega; \mathbb R^{n})]^m$. Then
\[
\begin{split}
&\int_\Omega \Big(\sum_{i=1}^m(\theta u_i + (1-\theta) v_i) \div w_i - \phi^*(x, |w|)\Big)\,dx\\
& \qquad = \theta\int_\Omega \Big(\sum_{i=1}^mu_i \div w_i - \phi^*(x, |w|)\Big)\,dx
+(1-\theta)\int_\Omega\Big(\sum_{i=1}^m v_i \div w_i - \phi^*(x, |w|)\Big)\,dx\\
&\qquad \le \theta \rho^m_{V,\phi}(u) + (1-\theta) \rho^m_{V,\phi}( v).
\end{split}
\] 
The claim follows when we take the supremum over such $w$. 

Finally, we show that $\rho^m_{V,\phi}$ is left-continuous. 
Since $\lambda\mapsto \rho^m_{V,\phi}(\lambda u)$ is increasing, 
$\rho^m_{V,\phi}(\lambda u) \le \rho^m_{V,\phi}(u)$ for $\lambda \in(0, 1)$.
We next consider the opposite inequality at the limit.
Let first $\rho^m_{V,\phi}(u)<\infty$ and fix $\epsilon>0$. 
By the definition of $\rho^m_{V,\phi}$ there exists a test function 
$w\in [C^1_0(\Omega; \mathbb R^{n})]^m$ with $\rho_{\phi^*}(|w|)<\infty$ such that
\[
\int_\Omega \Big( \sum_{i=1}^mu_i \div w_i \Big)\,dx 
\ge 
\rho^m_{V,\phi}(u) - \epsilon + \rho_{\phi^*}(|w|). 
\]
Multiplying this inequality by $\lambda\in (0,1)$ and subtracting $\rho_{\phi^*}(|w|)$, 
we obtain that 
\[
\rho^m_{V,\phi}(\lambda u) 
\ge 
\lambda (\rho^m_{V,\phi}(u) - \epsilon) + (\lambda-1)\rho_{\phi^*}(|w|). 
\]
Hence 
\[
\lim_{\lambda\to 1^-}\rho^m_{V,\phi}(\lambda u) 
\ge 
\rho^m_{V,\phi}(u) - \epsilon. 
\]
The claim follows from this as $\epsilon\to 0^+$. The case $\rho^m_{V,\phi}(u)=\infty$ is proved similarly, 
we only need to replace $\rho^m_{V,\phi}(u) - \epsilon$ by $\frac1\epsilon$. 
\end{proof}

\begin{lemma}[$BV$-type approximation by smooth functions]\label{lem:density0}
Assume that $\phi \in \Phiw(\Omega)$ satisfies \azero{}, \aone{} and \adec{}. 
Then there exists $c\ge 1$ 
such that for every $u \in L^\phi(\Omega;\R^m)$ we can find $u_k \in C^\infty(\Omega;\R^m)$ with 
\[
u_k \to u\text{ in }L^\phi(\Omega;\R^m)
\quad\text{and}\quad
V^m_\phi(u) \le \lim_{k \to \infty} V^m_\phi(u_k)\le c V^m_\phi(u).
\]
\end{lemma}

The following result is the vectorial counterpart of \cite[Theorem~5.2]{EleHH25}

\begin{lemma}\label{lem:BV-gradient}
Let $\phi \in \Phiw(\Omega)$ and $u \in W^{1, 1}_\loc(\Omega;\R^m)$. 
Then $V^m_\phi(u) \le \| \nabla u \|_{(L^{\phi^*}(\Omega;\R^m))'}$.
\begin{enumerate}
\item[(1)]
If $C^1_0(\Omega; \mathbb R^{m\times n})$ is dense in $L^{\phi^*}(\Omega; \mathbb R^{m\times n})$, 
then $V^m_\phi(u) = \| \nabla u \|_{(L^{\phi^*}(\Omega;\R^m))'}$.
\item [(2)]
If $\phi$ satisfies \azero{}, \aone{} and \adec{}, 
then $V^m_\phi(u) \approx \| \nabla u \|_{L^\phi(\Omega;\R^m)}$.
\end{enumerate}
\end{lemma}
\begin{proof}
Since $u \in W^{1, 1}_\loc(\Omega;\R^m)$, it follows from the 
definition of $V^m_\phi$ and integration by parts that 
\begin{equation}\label{eq:integrationByParts}
V^m_\phi(u)=\sup\bigg\{ \int_\Omega \sum_{i=1}^m\nabla u_i \cdot w_i \,dx \colon  w \in [C^1_0(\Omega; \mathbb R^n)]^m, \| w \|_{\phi^*}\le 1 
\bigg\},
\end{equation}
The definition of the associate space norm implies that 
\[
\int_\Omega \sum_{i=1}^m\nabla u_i \cdot w_i \,dx
\le 
\int_\Omega |\nabla u|\, |w| \,dx
\le 
\| \nabla u \|_{(L^{\phi^*}(\Omega;\R^m))'} \|w\|_{L^{\phi^*}(\Omega;\mathbb R^{m\times n})}.
\]
Taking the supremum over $w\in [C^1_0(\Omega; \mathbb R^n)]^m$ with $\|w\|_{L^{\phi^*}(\Omega)}\le 1$,
we conclude that 
$V^m_\phi(u) \le \| \nabla u \|_{(L^{\phi^*}(\Omega))'}$.

To prove claim (1), we next show the opposite inequality 
$\| \nabla u \|_{(L^{\phi^*}(\Omega;\R^m))'} \le V^m_\phi(u)$ when smooth functions are dense.
Let $ w \in L^{\phi^*}(\Omega; \mathbb R^{m\times n})$ with $\|w\|_{\phi^*}=1$ and let $(w_j)$ be a sequence from 
$C^1_0(\Omega; \mathbb R^{m \times n})$ with $w_j \to w$ in $L^{\phi^*}(\Omega; \mathbb R^{m\times n})$ and almost everywhere. 
Since also $w_j/\|w_j\|_{\phi^*} \to w$ in $L^{\phi^*}(\Omega; \mathbb R^{m\times n})$, we may assume that 
$\|w_j\|_{\phi^*}=1$. By Fatou's Lemma, 
\[
\liminf_{j \to \infty} \int_\Omega \sum_{i=1}^m\nabla u_i \cdot (w_j)_i \,dx \ge \int_\Omega \sum_{i=1}^m\nabla u_i \cdot w_i \,dx,
\]
so it follows from \eqref{eq:integrationByParts} that 
\[
V^m_\phi(u)\ge \sup\bigg\{ \int_\Omega \sum_{i=1}^m\nabla u_i \cdot w_i \,dx \,\Big|\, w \in L^{\phi^*}(\Omega; \mathbb R^{m\times n}), \| 
w\|_{\phi^*}\le 
1 \bigg\}.
\]
For $h\in L^{\phi^*}(\Omega)$ we set $w:=\frac{\nabla u}{|\nabla u|} h$ if $|\nabla u|\neq 0$ 
and $0$ otherwise in the inequality above. Thus 
\[
V^m_\phi(u)\ge \sup\bigg\{ \int_\Omega |\nabla u| \, h \,dx \,\Big|\, h \in L^{\phi^*}(\Omega), \| h\|_{\phi^*}\le 1 
\bigg\} = \| \nabla u \|_{(L^{\phi^*}(\Omega))'}. 
\]
Hence $V^m_\phi(u) = \| \nabla u \|_{(L^{\phi^*}(\Omega))'}$ and the proof of (1) is complete.

Then we prove (2).
Fix $h\in C^1_0(\Omega)$. Since $w_{\epsilon,\delta}:=(\frac{\nabla u}{|\nabla u|+\epsilon})* \eta_\delta$ is 
bounded and converges to $\frac{\nabla u}{|\nabla u|+\epsilon}$ a.e.\ as $\delta\to 0$, 
it follows by dominated convergence (with majorant $|\nabla u|\, |h|$) that 
\[
\lim_{\epsilon\to 0^+}\lim_{\delta\to 0^+}\int_\Omega \sum_{i=1}^m\nabla u_i \cdot (w_{\epsilon,\delta} h)_i \,dx
=
\int_\Omega |\nabla u|\, h \,dx.
\]
Since $w_{\epsilon,\delta}h\in C^1_0(\Omega;\mathbb R^{m\times n})$, this and \eqref{eq:integrationByParts} imply that 
\[
V^m_\phi(u)
\ge
\sup\bigg\{ \int_\Omega |\nabla u|\, h \,dx \,\Big|\, h \in C^1_0(\Omega), \| h \|_{\phi^*}\le 1 \bigg\}. 
\]
The final part of the proof coincides verbatim with \cite[Theorem~5.2]{EleHH25}, so we omit it.
\end{proof}

Note that a sufficient condition for the density of $C^1_0(\Omega; \mathbb R^{m\times n})$ in 
$L^{\phi^*}(\Omega; \mathbb R^{m\times n})$ is that $\phi$ satisfies \azero{} and \adec{} 
\cite[Theorem~3.7.5]{HarH19}.

In \cite{HarHL08, HarHLT13}, a BV-type space with variable exponent was defined based on the set 
$Y$ where the exponent equals $1$ using the modular which is denoted $\rho_{\text{old}}$ below. The next result connects this definition to the one 
given above based on \cite{EleHH25}.  

\begin{theorem}\label{thm:equivBVs}
Let $p:\Omega \to [1, \infty)$ be lower semicontinuous 
and strongly $\log$-Hölder continuous in $Y$. 
Then 
\[
\BV^\px(\Omega;\R^m) =\BV(\Omega;\R^m)\cap W^{1, p(\cdot)}(\Omega \setminus Y;\R^m)
\] 
and $V^m_{\px}(u; \Omega)\approx \|u\|_{\rho_{\text{\rm old}}}$, where 
$\rho_{\text{\rm old}}(u) := |D u|(Y) + \rho_{L^{p(\cdot)}(\Omega \setminus Y; \R^{n\times m})}(\nabla u)$.
\end{theorem}
\begin{proof}
By Lemma~\ref{BVvectorial}, we can argue using components and so it suffices to 
consider the case $m=1$. 
Furthermore, it suffices to prove that $\|\cdot\|_{\rho_{V,\phi}}\approx \|u\|_{\rho_{\text{\rm old}}}$ since $V^m_\phi(u; \Omega) \approx \|\cdot\|_{\rho_{V,\phi}}$ by \cite[Lemma~4.7]{EleHH25} 
and the Lebesgue  norms of the function are the same in both spaces. 

Assume first that $u\in \BV(\Omega)\cap W^{1, p(\cdot)}(\Omega \setminus Y)$ and 
$\rho_{\text{\rm old}}(u)<\infty$. 
If $\rho_{\phi^*}(|w|) < \infty$, then 
$|w|\le 1$ on $Y$ by Lemma~\ref{lem:smallOnY}. Thus 
\[
\begin{split}
\rho_{V,\px}(u)
&= \sup \bigg\{ \int_{\Omega} u \div w \,dx- \int_{\Omega \setminus Y}\frac{|w|^{p'(x)}}{p'(x)} \,dx \,\Big|\,  w \in C^1_0(\Omega, \R^n), \rho_{\phi^*}(|w|) < \infty \bigg\}\\
&\le \sup \bigg\{ \int_{\Omega} u \div w \,dx- \int_{\Omega \setminus Y}\frac{|w|^{p'(x)}}{p'(x)} \,dx 
\,\Big|\, w \in C^1_0(\Omega, \R^n), |w| \le 1 \text{ on } Y\bigg\}
\end{split}
\] 
where $p'$ is the H\"older conjugate of $p$. Since $p$ is lower semicontinuous, 
$Y = \{ p\le 1\}$ is closed.
Since further $u \in W^{1, p(\cdot)}(\Omega \setminus Y)$, $D^su=0$ in the open set 
$\Omega \setminus Y$. Then $|w|\le 1$ and Young's inequality imply that  
\[
\int_{\Omega} u \div w \,dx
=
\int_{\Omega \setminus Y} w \cdot \nabla u \,dx + \int_{Y} w \, dDu 
\le 
\int_{\Omega \setminus Y} \frac{|\nabla u|^{p(x)}}{p(x)}+ \frac{|w|^{p'(x)}}{p'(x)} \,dx
+ |Du|(Y).
\]
Using this in the previous estimate for the modular, we find that 
\[
\rho_{V,\px}(u) \le
\int_{\Omega \setminus Y} \frac{|\nabla u|^{p(x)}}{p(x)} \,dx  + |Du|(Y).  
\]
We have thus shown that $\rho_{V,\px}(u) \le \rho_{\text{old}}(u)$. 
It follows by the definition of the Luxemburg norm that
$\|\cdot\|_{\rho_{V,\px}} \le \|\cdot\|_{\rho_{\text{old}}}$.

Assume then conversely that $u\in \BV^\px(\Omega)$ so that $V_\px(u; \Omega)<\infty$.
Denote $F_i := \{p \le 1 + \frac1i\}$. Then $F_i \searrow Y$ 
and $p^-_{\Omega \setminus F_i} \ge 1 +\frac1i$. 
Since $p$ is lower semicontinuous, $F_i$ is closed in $\Omega$.
%
Thus $(p')^+_{\Omega \setminus F_i}<\infty$
so that $C^1_0(\Omega\setminus F_i)$ is dense in $L^{p'(\cdot)}(\Omega \setminus F_i)$ by \cite[Theorem~3.7.15]{HarH19}. Thus Lemma~\ref{lem:BV-gradient}(1) yields that 
\[
 \|\nabla u\|_{(L^{\phi^*}(\Omega \setminus F_i))'} = V_\px(u; \Omega \setminus F_i) \le
 V_\px(u; \Omega).
\] 
By \cite[Theorem 3.4.6 and Proposition 2.4.5]{HarH19}, 
$\|\nabla u\|_{L^{\phi^*}(\Omega\setminus F_i))'} \approx 
\|\nabla u\|_{L^{\phi} (\Omega \setminus F_i)}$ 
where the implicit constant $c_1$ does not depend on $p^-_{\Omega \setminus F}$.
The previous inequality for $v:= u / (c_1 V_\px(u; \Omega))$ and monotone convergence yield 
\[
\int_{\Omega\setminus Y} |\nabla v|^{p(x)} \,dx 
=
\lim_{i\to\infty}\int_{\Omega \setminus F_i} |\nabla v|^{p(x)} \,dx 
\le
\lim_{i\to\infty}\max\big\{1, 
\|\nabla v\|_{L^\px(\Omega \setminus F_i)}^{p^+}\big\}
\le 1.
\]
Hence $\|\nabla v\|_{L^{p(\cdot)}(\Omega \setminus Y)}\le 1$ and so $\|\nabla u\|_{L^{p(\cdot)}(\Omega \setminus Y)} \le c V_\px(u; \Omega)$.

By \cite[Example~4.2]{EleHH25} we have  $BV^{\px}(\Omega) \subset \BV(\Omega)$, and 
$|Du|(\Omega)\le c V_\px(u; \Omega)$ since $\phi$ satisfies $\azero{}$. 
Thus $\|u\|_{\rho_{\text{old}}} \lesssim V_\px(u; \Omega)$, 
which concludes the proof.
\end{proof}


\section{Proof of main result}

We recall the definition of the weak recession function
\[
f^\infty(\xi)=\limsup_{t\to\infty}\frac{f(t\xi)}{t}.
\]

\begin{remark}\label{rem:recession}
If $f:\mathbb R^{m\times n}\to[0,\infty)$ is convex, then the limit superior 
in $f^\infty$ is a limit as $t\mapsto \frac{f(t\xi)-f(0)}{t}$ is increasing. 
If, additionally, $f$ is Lipschitz, then 
\[
f^\infty_s (\xi) = 
\lim_{\substack{t\to \infty\\\xi'\to \xi}} \frac{f(t\xi')}{t}
=
\lim_{\substack{t\to \infty\\\xi'\to \xi}} \frac{f(t\xi) -f(t \xi')}{t} + f^\infty(\xi)
= f^\infty(\xi).
\]
In this case there is no need distinguish the weak and strong recession functions. 
As we recall below, our assumptions imply that $f^\infty_s = f^\infty$ for rank-one matrices, 
in particular for $f^\infty(\frac{d D^s u}{d |D^s u|})$.

However, the weak and strong recession functions are not the same without uniformly 
linear growth: if $x_i\in \Omega \setminus Y$ with $x_i\to x\in Y$, then 
\[
(f^{p(x)})^\infty_s (\xi) \ge  
\lim_{x_i\to x} \lim_{\substack{t\to \infty\\\xi'\to \xi}} \frac{f(t\xi')^{p(x_i)}}{t}
= +\infty 
> f^\infty(\xi) =(f^{p(x)})^\infty_w(\xi), 
\] 
provided $f$ satisfies \Hone{} and \Hthree{}. Thus we need to be careful when considering 
recession functions in the absence of uniformly linear growth. 
%
\end{remark}

\begin{proposition}\label{prop:Fqc}
Let $\Omega$ be a bounded open set with Lipschitz boundary and 
$p$ be strongly $\log$-Hölder continuous.
If $f:\mathbb R^{m\times n}\to [0,\infty)$ satisfies \Hzero{}, \Hone{}, \Htwo{} and \Hthree{}, then the functional $F_{qc} : \BV^{p(\cdot)}(\Omega;\R^m) \to [0, \infty)$, defined as 
\[
F_{qc}(u)
:=
\int_\Omega f(\nabla u)^{p(x)}\,dx + 
\int_Y f^\infty\Big(\frac{d D^s u}{d |D^s u|}\Big)\, d |D^s u|,
\]
is lower semicontinuous with respect to $L^\px(\Omega;\R^m)$-convergence.
\end{proposition}

\begin{proof}
For $j \in \mathbb N$, $x \in \Omega$ and $t \in [0,\infty)$, define 
\[
\phi_j(x, t):=\Big\{\begin{array}{ll} t^{p(x)} & \hbox{ if } t \le j,\\
j^{p(x)}+ p(x)j^{p(x)-1}(t-j), & \hbox{ otherwise.}
\end{array}
\Big.
\]
Observe that $\phi_j$ is continuous,
$\phi_j(x,t) \le t^{p(x)}$, $\phi_j(x,t)\nearrow t^{p(x)}$ as $j \to \infty$, and 
\[
t-1 \le \phi_j(x,t) \le j^{p^+}+  p^+j^{p^+-1}(t-j).
\]
Let us write  $\psi_j(x,\cdot):=\phi_j(x, f(\cdot))$ and  note it is quasiconvex as the composition of a convex 
function and a quasiconvex function, by Jensen's inequality. By \Htwo{} and \Hthree{} we find that
\[
\begin{split}
m |\xi| -1 \le f(\xi) -1 \le \psi_j(x,\xi) &\le j^{p^+}+  p^+j^{p^+-1}f(\xi) \\
&\le j^{p^+}+  Mp^+j^{p^+} (1+|\xi|)
\le Mp^+j^{p^+} (2 + |\xi|).
\end{split}
\]
We consider the strong recession function (of $\psi_j+1$)
\[
\Psi_j (x, \xi) := (\psi_j+1)^\infty_s(x, \xi) =
\lim_{\substack{t\to \infty\\\xi'\to \xi\\ x' \to x}} \frac{\psi_j(x', t\xi')+1}{t},
\]
whenever the limit exists.
If $f(\xi)\leq j$, then
\[
\begin{split}
|\psi_j(x,\xi)-\psi_j(y,\xi)| &\leq |f(\xi)^{p(x)}-f(\xi)^{p(y)}|\leq j^{p^+}|1-j^{|p(x)-p(y)|}|.
\end{split}
\]
If $f(\xi)>j$, then by \Hthree{}
\begin{align*}
	|\psi_j(x,\xi)-\psi_j(y,\xi)|&\leq |j^{p(x)}-j^{p(y)}|+ (f(\xi)-j)\,|p(x)j^{p(x)-1}- p(y)j^{p(y)-1}|\\
	&\leq j^{p^+}|1-j^{|p(x)-p(y)|}|+ M (1+ |\xi|)\,|p(x)j^{p(x)-1}- p(y)j^{p(y)-1}|.
\end{align*}
Since $p \in C(\overline \Omega)$ we obtain that
$|\psi_j(x,\xi)-\psi_j(y,\xi)| \leq \omega_j (|x-y|) (1+ |\xi|)$ for  suitable moduli of continuity $\omega_j$. Moreover, by assumption, $\Omega$ is a bounded domain with Lipschitz boundary. 
We have thus verified all assumptions of Lemma~\ref{lem:lsc}, and so 
\[
u\mapsto 
\int_\Omega \psi_j(x,\nabla u)+1\,dx + \int_\Omega \Psi_j\Big(x,\frac{d D^s u}{d |D^s u|}\Big) \, d |D^s u|
\]
is  weakly* lower semicontinuous in $\BV(\Omega;\R^m)$. We observe that the ``$+1$'' in the 
integral does not affect the claim (it was only used to satisfy the assumptions of the lemma). 
By \cite[Lemma~6.1]{BreDG20}, $\xi \mapsto \psi_j(x, \xi)$ is Lipschitz continuous, and so 
\[
\Bigg|\lim_{\substack{t\to \infty\\\xi'\to \xi\\x' \to x}} \frac{\psi_j(x', t\xi')+1}{t}
- 
\lim_{\substack{t\to \infty\\ x' \to x}} \frac{\psi_j(x', t\xi)}{t}
\Bigg|
\le 
\lim_{\substack{t\to \infty\\\xi'\to \xi\\x' \to x}} \frac{|\psi_j(x', t\xi) - \psi_j(x', t\xi')|}{t} = 0.
\]
By the definition of $\psi_j$, the continuity of $p$ and Remark~\ref{rem:recession}, 
we conclude that
\[
\Psi_j (x, \xi)
=\lim_{\substack{t\to \infty\\ x' \to x}} \frac{\psi_j(x', t\xi)}{t}
=\lim_{x' \to x} p(x') j^{p(x')-1} \lim_{t\to \infty} \frac{ f(t\xi)}{t}
= p(x) j^{p(x)-1} f^\infty(\xi)
\]
for every rank-one matrix $\xi \in \mathbb R^{m\times n}$. If $x\in \Omega\setminus Y$, 
the right-hand side tends to $+\infty$, otherwise it equals $f^\infty$. 
Hence $\Psi_j (x, \xi) \nearrow (f^{p(x)})^\infty(\xi)$ as $j \to \infty$, for every rank-one matrix $\xi \in \mathbb R^{m\times n}$, in particular, by Alberti's theorem \cite{A93}, for 
$\xi=\frac{d D^s u}{d |D^s u|}$. 

Applying the lower semicontinuity result, we obtain 
\begin{align*}
&\int_\Omega \psi_j(x, \nabla u)\,dx + \int_\Omega \Psi_j\Big(x,  \frac{dD^s u}{d |D^s u|}\Big) d |D^s u| \\
&\qquad\le 
\liminf_{h \to \infty} \bigg(\int_\Omega \psi_j(x, \nabla u_h)\,dx + \int_\Omega \Psi_j\Big(x, \frac{dD^s u_h}{d |D^s u_h|} \Big)\, d |D^s u_h|\bigg)\\
&\qquad\le 
\liminf_{h \to \infty} \bigg(\int_\Omega f(\nabla u_h)^{p(x)}\,dx + \int_\Omega (f^{p(x)})^\infty\Big(\frac{d D^s u_h}{d |D^s u_h|} \Big)\, d |D^s u_h|\bigg), 
\end{align*}
whenever $u_h \overset{\ast}{\rightharpoonup}u$ in $BV(\Omega;\R^m)$, where 
we used  $\psi_j (x, \xi) \le f(\xi)^{p(x)}$ and $\Psi_j(x, \xi)\le (f^{p(x)})^\infty$ in the second step. 
Now the right hand side is independent of $j$, and hence we obtain 
by monotone convergence as $j \to \infty$,  that 
\begin{align*}
F_{qc}(u)&=\int_\Omega f(\nabla u)^{p(x)}\,dx + \int_\Omega (f^{p(x)})^\infty\Big(  \frac{dD^s u}{d |D^s u|}\Big)\, d |D^s u| \\
&\le 
\liminf_{h\to \infty}\bigg(\int_\Omega f(\nabla u_h)^{p(x)}\,dx + 
\int_\Omega (f^{p(x)})^\infty\Big(\frac{dD^s u_h}{d |D^s u_h|}\Big)\, d |D^s u_h|\bigg)
\end{align*}
whenever $u_h \overset{\ast}{\rightharpoonup}u$ in $BV(\Omega;\R^m)$.

Suppose next that $u_h\to u$ in $L^\px(\Omega;\R^m)$ for $u, u_h\in \BV^\px(\Omega; \R^m)$. 
 We may assume that 
\[
\liminf_{h\to \infty} F_{qc}(u_h) < \infty
\]
and choose a subsequence with $\lim_{k\to \infty} F_{qc}(u_{h_k})=\liminf_{h\to \infty} F_{qc}(u_h)$. 
We denote the subsequence by $u_h$, again. 
By \Htwo{}, taking into account that $\Omega$ is bounded 
\begin{align*}
\lim_h |Du_h|(\Omega) 
&= \lim_h \left(\int_\Omega |\nabla u_h|\,dx + |D^s u_h|(\Omega)\right)\\
&\le \lim_h \left( \int_\Omega 1+ |\nabla u_h|^{p(x)} dx + |D^s u_h|(\Omega)\right)\\
&\lesssim 
\lim_h \left(\int_\Omega f(\nabla u_h)^{p(x)} dx + \int_\Omega f^\infty\Big(\frac{d D^s u_h}{d |D^s u_h|}\Big)\, d |D^s u_h|\right)
<\infty.
\end{align*}
By compactness, $u_h$ converges to $u$ weakly* in $\BV$, up to a subsequence again denoted $u_h$ 
(see \cite[Proposition~1.59]{AmbFP00}).
Applying the previous result to this subsequence, we obtain that
\begin{align*}
F_{qc}(u)
\le 
\liminf_{h\to \infty}\bigg(\int_\Omega f(\nabla u_h)^{p(x)}\,dx + 
\int_\Omega (f^{p(x)})^\infty\Big(\frac{dD^s u_h}{d |D^s u_h|}\Big)\, d |D^s u_h|\bigg)
= \liminf_{h\to \infty} F_{qc}(u_h)
\end{align*}
since $(f^{p(x)})^\infty = f^\infty$ in $Y$, which includes the support of $D^su_h$ by Theorem~\ref{thm:equivBVs}.
\end{proof}

\begin{remark}
Suppose that $u, u_h \in W^{1,p(\cdot)}(\Omega;\R^m)$ and $u_h \to u$ in $L^{p(\cdot)}(\Omega;\R^m)$.
The previous proposition implies that 
\[
\int_\Omega f(\nabla u)^{p(x)} \,dx \le \liminf_h \int_\Omega f(\nabla u_h)^{p(x)}\,dx.
\]
\end{remark}

We remark that the continuity of $f$ is not actually needed in the next proposition, it suffices 
to assume that it is Borel.  

\begin{proposition}\label{prop:Fcalmeas}
Let $p$ be strongly $\log$-Hölder continuous and 
$f:\R^{n \times m} \to [0,\infty)$ satisfy \Hzero{}, \Htwo{} and \Hthree{}. 
Then
\begin{align}\label{estFcal}
\mathcal F(u,A) \le C\big(|A| + V^m_\px(u; A)+ V^m_\px(u; A)^{p^+}\big)
\end{align}
for every $u \in \BV^{p(\cdot)}(\Omega;\R^m)$ and open $A \subset \Omega$.
 
Moreover, for fixed $u \in \BV^{p(\cdot)}(\Omega;\R^m)$, the set function $A\mapsto \mathcal F(u,A)$ 
is the restriction to open subsets of $\Omega$ of a finite Radon measure in $\Omega$. 
\end{proposition}
\begin{proof}
Since $p$ is $\log$-Hölder continuous, we can choose functions 
$u_h \in W^{1,\px}(A;\R^m)\cap C^\infty(A; \R^m)$ with $u_h \to u$ in $L^\px(A; \R^m)$ and 
\[
\|\nabla u_h\|_{L^\px(A;\R^m)} \approx 
\lim V^m_\px(u_h; A)\le c V^m_\px(u; A)
\] 
by Lemmas~\ref{lem:density0} and \ref{lem:BV-gradient}(2). 
Assumption \Hthree{} and \cite[Lemma~3.2.9]{HarH19} imply that 
\[
\begin{split}
\mathcal F(u,A) 
&\le \liminf_h \int_A f(\nabla u_h)^{p(x)}\,dx \le (2M)^{p^+} \liminf_h \int_A 1 + |\nabla u_h|^{p(x)}\,dx \\
&\le \liminf_h \Big({\mathcal L}^n (A)+ \max\Big\{\|\nabla u_h\|_{L^\px(A;\R^m)}, 
\|\nabla u_h\|_{L^\px(A;\R^m)}^{p^+}\Big\} \Big)\\
& \le C(p^+) \big({\mathcal L}^n (A) + V^m_\px(u; A)+ V^m_\px(u; A)^{p^+}\big).
\end{split}
\] 
This is the desired inequality. 

It remains to show that $\mathcal F(u,\cdot)$ is the restriction of a Radon measure for every $u\in \BV^{p(\cdot)}(\Omega;\R^m)$.
We will establish this by checking the 
De Giorgi--Letta criteria: 
every increasing, bounded set function $\mu: \mathcal A \to[0,\infty)$ defined on the open subsets 
$\mathcal A$ of $\Omega$ such that
\begin{itemize}
\item[(i)] $A, A' \in {\mathcal A}\ \Rightarrow \ 
\mu(A\cup A')\le \mu(A) + \mu(A')$,
\item[(ii)] $A, A' \in {\mathcal A}, A \cap A' =\emptyset \ \Rightarrow \ 
\mu (A\cup A')\geq \mu (A)+\mu(A')$, 
\item[(iii)] $\mu(A)=\sup\{\mu(A') \mid A'\in \mathcal A, A'\subset \subset A\},$
\end{itemize}
can be extended to Borel sets $B\subset \Omega$ by
$\mu(B)=\inf\{\mu(A) \mid B\subset A \in \mathcal A\}$ as a positive Radon 
measure \cite[Theorem~1.53]{AmbFP00}.

If $u\in \BV^{p(\cdot)}(\Omega;\R^m)$, then the previously shown inequality ensures that 
$\mathcal F(u,A)$ is finite; since $f\ge 0$, it is increasing and non-negative. 
The superadditivity condition (ii) is a consequence of the definition of $\mathcal F$ 
since $u_h\in W^{1,\px}(A\cup A'; \R^m)$ if and only if 
$u_h|_{A}\in W^{1,\px}(A; \R^m)$ and $u_h|_{A'}\in W^{1,\px}(A'; \R^m)$ 
for disjoint and open $A$ and $A'$. 

Let $A_2 \subset \subset A_1 \subset  A$. 
We will prove that 
\begin{align}\label{claim1}
\mathcal F(u,A) \le 
\mathcal F(u,A_1) + \mathcal F(u, A \setminus \overline A_2),
\end{align}
but first we show how this inequality implies (i) and (iii). 
By \eqref{estFcal} and Theorem~\ref{thm:equivBVs}, we obtain
\[
\mathcal F(u,A) 
\le 
C\big({\mathcal L}^n (A) + V^m_\px(u; A)+ V^m_\px(u; A)^{p^+}\big)
\le 
C\big(|A| + \|u\|_{\rho_{\text{\rm old}}, A}+ \|u\|_{\rho_{\text{\rm old}}, A}^{p^+}\big).
\]
Since $p^+<\infty$, the modular $\rho_{\text{\rm old}}(u)$ over $A$ is finite, and hence from the definition of $\rho_{\text{\rm old}}$ we obtain that $ \|u\|_{\rho_{\text{\rm old}}, F} \to 0$ as $F \to \emptyset$. 
Since $A$ is open, $A \setminus \bar A_2 \to \emptyset$ as  $A_2 \nearrow A$.
Then (iii) follows by $\mathcal F(u,A) \le \mathcal F(u,A_1) + \mathcal F(u, A \setminus \overline A_2)$ as $A_2 \nearrow A$.

To show (i), we replace $A$ by $A\cup A'$ in \eqref{claim1},  
choose $A_1 \subset \subset A$ and use monotonicity:
\[
\mathcal F(u, A \cup A')
\le 
\mathcal F(u, A_1)+ \mathcal F(u, (A \cup A')\setminus \overline A_2)
\le 
\mathcal F(u,A)+ \mathcal F(u, (A \cup A')\setminus \overline A_2).
\]
Then (i) follows by (iii) as $A_2\nearrow A$.

It remains to prove \eqref{claim1}.
By the definition of $\mathcal F$ and a diagonal argument we choose sequences 
$\{u_{1,h}\}\subset  W^{1,\px}(A_1;\R^m)$ and 
$\{u_{2,h}\}\subset W^{1,\px}(A \setminus \overline A_2;\R^m)$ that converge to 
$u$ in $L^1(A_1;\R^m)$ and $L^1(A \setminus \overline A_2;\R^m)$, respectively, with
\begin{equation}\label{1}
\lim_ h \int_{A_1} f(\nabla u_{1,h})^{p(x)}\,dx = \mathcal F(u, A_1)
\quad
\text{and} 
\quad
\lim_ h \int_{A \setminus \overline A_2} f(\nabla u_{2,h})^{p(x)}\,dx = \mathcal F(u, A \setminus \overline A_2).
\end{equation}

For $t \in (0, 1)$, let $A_t := \{x \in A_1 \mid {\rm dist}(x, \partial A_1) > t\}$. 
Let $\eta \in (0, \dist(\partial A_1, \partial A_2))$ to be chosen later. 
Consider a smooth cut-off function $\zeta_\delta \in  C^\infty_0(A_{\eta-\delta};[0,1])$
such that $\zeta_\delta = 1$ on $A_\eta$ and
$\|\nabla \zeta_\delta\|_{L^\infty(A_{\eta-\delta})}\le \frac{C}\delta$.
The sequence $v_{h,\delta}:= \zeta_\delta  u_{1,h}+ (1-\zeta_\delta)u_{2,h}$ 
converges to $u$ in $L^\px(A;\R^m)$ as $h\to\infty$. 
Since $\nabla \zeta_\delta \ne 0$ only in $A_{\eta-\delta}\setminus A_\eta$, we obtain
\begin{align*}
\int_{A} f(\nabla v_{h,\delta})^{p(x)}\,dx 
& \le \int_{A_{\eta}}f(\nabla u_{1,h})^{p(x)}\,dx + \int_{A \setminus A_{\eta -\delta}} f(\nabla u_{2,h})^{p(x)}\,dx \\
& \qquad + \int_{A_{\eta-\delta}\setminus A_\eta}  f(\nabla v_{h,\delta})^{p(x)}\,dx.  
\end{align*}
In the first two integral on the right the set can be increased to $A_1$ and $A\setminus \overline 
{A_2}$ since $f\ge 0$. 
For the last term we use \Hthree{} and 
$\nabla v_{h,\delta}= \zeta_\delta  \nabla u_{1,h}+ (1-\zeta_\delta) \nabla u_{2,h} +(u_{1,h}- u_{2, h}) \nabla \zeta_\delta$ to obtain
\begin{align*}
f(\nabla v_{h,\delta})^{p(x)}
&\le  M^{p(x)} (1 + |\nabla v_{h,\delta}|)^{p(x)}  \\
&\le  C(p^+) \big(1 + |\nabla u_{1,h}|^{p(x)}+ |\nabla u_{2,h}|^{p(x)}+ (\tfrac{1}\delta|u_{1,h}- u_{2,h}|)^{p(x)}\big)
\end{align*}
Thus 
\begin{align*}
\int_{A} f(\nabla v_{h,\delta})^{p(x)}\,dx 
& \le \int_{A_1}f(\nabla u_{1,h})^{p(x)}\,dx + \int_{A \setminus A_2} f(\nabla u_{2,h})^{p(x)}\,dx 
\\
&\qquad + C(p^+) \nu_h(A_{\eta-\delta}\setminus A_\eta)
+ \int_{A_1\setminus A_2} (\tfrac{1}\delta|u_{1,h}- u_{2,h}|)^{p(x)}\, dx. 
\end{align*}
where, for $E \subset \R^n$, 
\[ 
\nu_h(E):= \int_{E \cap (A_1 \setminus A_2)}  1 + |\nabla u_{1,h}|^{p(x)}+ |\nabla u_{2,h}|^{p(x)} \,dx. 
\] 
Passing to the limit as $h \to \infty$ and observing that $u_{1,h}- u_{2,h} \to u-u=0$ in $L^\px(A_1 \setminus A_2;\R^m)$, we obtain that 
\begin{align*}
\mathcal F(u, A)&\le \mathcal F(u, A_1) + \mathcal F(u, A \setminus \overline A_2)
+C(p^+) \limsup_{h \to 0} \nu_h(A_{\eta-\delta}\setminus A_\eta).
\end{align*}
We complete the proof of \eqref{claim1} by showing that the last term tends to 
zero when $\delta\to 0$. 

By the properties of Lebesgue integral $\nu_h$ is a Radon measure. By \eqref{1} and \Htwo{}, 
$\sup_h \nu_h(K) <\infty$ for every compact $K \subset \R^n$.
By the weak compactness of measures \cite[Theorem~1.41]{EvaG15}, there exists a subsequence, still denoted by $(\nu_h)$, and a Radon measure 
$\nu$ such that
\[
\limsup_{h \to \infty} \nu_h(K) \le \nu(K)
\]
for all compact sets $K \subset \Rn$, so that in particular
\[
\limsup_{h \to 0} \nu_h(A_{\eta-\delta}\setminus A_\eta) \le 
\nu (\overline{A_{\eta-\delta}\setminus A_\eta}).
\]
Let $(\delta_i)$ be a positive sequence converging to zero.
Since $\overline{A_{\eta-\delta_i}\setminus A_\eta}$ is a decreasing sequence of sets, we obtain
\[
\limsup_{i \to \infty} \nu(\overline{A_{\eta-\delta_i}\setminus A_\eta})= \nu\Big(\bigcap_{i=1}^\infty \overline{A_{\eta-\delta_i}\setminus A_\eta}\Big ) = \nu(\partial A_\eta).  
\]
Since $\nu(\cup_{\eta>0} \partial A_\eta \cap A_1)\le \nu(A_1)<\infty$, 
we can fix $\eta$ such that $\nu(\partial A_\eta)=0$.
\end{proof}

\begin{lemma}\label{lem:ubrelaxation}
Let $\Omega$ be a bounded open set with Lipschitz boundary, let $p$ be strongly $\log$-Hölder continuous, and 
let $f$ satisfy \Hzero{}, \Hone{}, \Htwo{} and \Hthree{}. Then
\[
\mathcal F(u,\Omega)\le 
\int_\Omega f(\nabla u)^{p(x)} \,dx +
\int_\Omega f^\infty\Big(\frac{d D^s u}{d |D^s u|}\Big) \, d|D^s u|
\]
for all $u \in \BV^\px(\Omega; \R^m)$.
\end{lemma}
\begin{proof}
We adapt the strategy of \cite[Proposition 5.49]{AmbFP00}. 
Recall that $Y:=\{p=1\}$ and denote $E^\delta:=\{x \in \mathbb R^n: \dist(x, E)\le  \delta\}$
for $E \subset \mathbb R^n$ and $\delta >0$.

Fix $u \in \BV^{p(\cdot)}(\Omega;\R^m)$. Let $\delta>0$ and consider the standard mollifiers 
$\eta_\delta$, and denote $u_\delta:= u * \eta_\delta$, so that
$\nabla u_\delta:= \nabla u \ast \eta_\delta+ D^su \ast \eta_\delta$.
Let $A \subset \subset \Omega$ be open and 
observe that $p(x) >1$ in 
$A \setminus Y^\delta$ and thus also $\nabla u_\delta= \nabla u \ast \eta_\delta$. 
Hence 
\[
\int_A f(\nabla u_\delta)^{p(x)}\,dx \le 
\int_{A \setminus Y^\delta} f(\nabla u*\eta_\delta)^{p(x)}\,dx 
+ \int_{Y^\delta \cap A } f(\nabla u_\delta)^{p(x)}\,dx. 
\]
We estimate the two integrals on the right-hand side in the next two paragraphs. 

Since $f$ satisfies \Hone{} and \Hthree{},  
\begin{align}\label{fpxLip}
\big|f(\xi)^{q}- f(\eta)^{q}\big|\le C(\bar q,M)(1+ |\xi|^{q-1}+ |\eta|^{q-1})|\xi-\eta|,
\end{align}
for every $\bar q\ge 1$, $q\in [1, \bar q]$ and $\xi, \eta \in \mathbb R^{n\times m}$ \cite[Lemma~5.42]{AmbFP00}.
Using this with $q=p(x)$, we find by Hölder's inequality that 
\begin{align*}
&\int_{A \setminus Y} \big|f(\nabla u)^{p(x)}-f(\nabla u *\eta_\delta)^{p(x)}\big|\,dx \\
&\qquad\lesssim
\int_{A \setminus Y} (1+|\nabla u|^{p(x)-1}+|\nabla u *\eta_\delta|^{p(x)-1})
\,|\nabla u -\nabla u *\eta_\delta|\,dx \\
&\qquad\lesssim
\big(1+\|\nabla u\|_\px + \|\nabla u *\eta_\delta\|_\px\big) \|\nabla u-\nabla u *\eta_\delta\|_\px.
\end{align*}
Since $\nabla u*\eta_\delta \to \nabla u$ in $L^\px(A\setminus Y; \R^m)$ 
\cite[Theorem~4.6.4]{DieHHR11}, the right-hand side above converges to zero as $\delta\to 0$. 
Thus
\[
\limsup_{\delta \to 0} \int_{A \setminus Y^\delta} f(\nabla u*\eta_\delta)^{p(x)}\,dx 
\le 
\limsup_{\delta \to 0} \int_{A \setminus Y} f(\nabla u *\eta_\delta)^{p(x)}\,dx 
=\int_{A \setminus Y} f(\nabla u)^{p(x)}\,dx.
\]

We then consider the second integral, over $Y^\delta\cap A$. 
There exists $c>0$ such that $|\nabla u_\delta|\le \|u\|_1 \sup |\nabla \eta_h|
\le c \delta^{-1}$. Hence,  by \Hthree{}, 
\[
f(\nabla u_\delta)^{p(x)}
= f(\nabla u_\delta)^{p(x)-1} f(\nabla u_\delta)
\le
(c\delta^{-1})^{p(x)-1} f(\nabla u_\delta)
\]
for all small $\delta$. Denote
\[
\omega(\delta):=\sup_{x\in Y^\delta} (p(x)-1)\log\frac1{\dist(x, Y)}.
\]
By the strong $\log$-Hölder assumption, $\omega(0^+)=0$ and so $(c\delta^{-1})^{p(x)-1}
\le e^{2(p(x)-1)\log 1/\delta}
\le e^{2\omega(\delta)}\to 1$
when $\delta\le \frac1c$. 
Hence 
\[
\int_{Y^\delta\cap A} f(\nabla u_\delta)^{p(x)}\,dx
\le e^{2\omega(\delta)}\int_{Y^\delta\cap A}f(\nabla u_\delta) \,dx.
\]
Now estimating the integral on the right hand side as in the proof of \cite[Theorem~5.49]{AmbFP00}, 
i.e.\ using \eqref{fpxLip} with $q=1$ and \cite[Theorem~2.2(b)]{AmbFP00}, we have
\begin{align*}
&\limsup_{\delta\to 0}\int_{Y^\delta\cap A} f(\nabla u_\delta)\,dx\\
&\qquad\le 
\limsup_{\delta\to 0}\bigg[\int_{Y^\delta\cap A} f(\nabla u \ast \eta_\delta)\,dx + M\,|D^s u \ast \eta_\delta|(Y^\delta\cap A)\bigg]\\
&\qquad\le 
\limsup_{\delta\to 0} \bigg[\int_{(Y^\delta\cap A)^\delta} f(\nabla u) + c\,|\nabla u*\eta_\delta-\nabla u|\,dx + M\,|D^s u|\big((Y^\delta\cap A)^\delta\big)\bigg] \\
&\qquad= 
\int_{Y\cap A} f(\nabla u) \,dx  + M\,|D^s u|(\overline A).
\end{align*}
Since $p=1$ in $Y$, the integrand on the right can be replaced by $f(\nabla u)^{p(x)}$.

Combining the estimates over $A\setminus Y^\delta$ and $Y^\delta\cap A$ and assuming 
that $|D^s u|(\partial A)=0$ we obtain that 
\begin{equation}\label{eq:acEst}
\mathcal F(u,A)\le \int_A f(\nabla u)^{p(x)}\,dx + C\,|D^s u|(A).
\end{equation}
From the proof of Proposition~\ref{prop:Fcalmeas} we know that 
$\mathcal F(u,A) = \sup_{B\subset A \text{ open}}\mathcal F(u,B)$. Since 
the sets $B$ can be chosen with $|D^s u|(\partial B)=0$ (as in the proof of the proposition), 
\eqref{eq:acEst} holds for all open $A\subset \Omega$. By the proposition, 
$\mathcal F(u,\cdot)$ extends to a Radon measure and it follows from its Lebesgue 
decomposition $\mathcal F=\mathcal F^a+\mathcal F^s$ that 
\begin{align*}
\mathcal F^a(u,B)\le \int_B f(\nabla u)^{p(x)}\,dx
\quad\text{and}\quad
\mathcal F^s(u,B)\le C\,|D^s u|(A)
\end{align*}
for every Borel set $B$, since $D^su$ and the Lebesgue measure are mutually singular. 

We still need a better estimate for $\mathcal F^s$; we use a similar strategy but estimate 
$\int_{Y^\delta\cap A} f(\nabla u_\delta)\,dx$ in a different way. 
Define $g(\xi):= \sup_{t >0}\frac{f(t\xi)- f(0)}{t} = \sup_{t >0}\frac{f(t\xi)}{t}$. 
Since $f$ is Lipschitz \cite[Lemma~5.42]{AmbFP00}, it follows that $g$ is, as well. 
By \Htwo{} and \Hthree{}, $m\,|\xi|\le g(\xi)\le M\,|\xi|$. 
We note that $f(\xi)\le g(\xi)=g(\frac \xi{|\xi|})|\xi|$ and 
use Lemma~\ref{lem:Reshetnyak} (Reshetnyak continuity) 
for the continuous and bounded function $g$ to conclude that
\begin{align*}
\liminf_{\delta\to 0} 
\int_{Y^\delta\cap A} f(\nabla u_\delta)\,dx 
&\le
\limsup_{\delta\to 0} \int_{Y^\delta\cap A} g\Big(\frac{\nabla u_\delta}{|\nabla u_\delta|}\Big)\, |\nabla u_\delta|\,dx 
= 
\int_{Y^{\delta_1}\cap A}g\Big(\frac{d D u}{d |D u|} \Big)\,d |D u| 
\end{align*}
for any $\delta_1>0$. 
We next use that $d|D u|= |\nabla u| \,dx + d|D^su|$ and that 
$\frac{d Du}{d |Du|}=\frac{d D^s u}{d |D^s u|}$ in the support of $D^su$ which is 
a subset of $Y$. Thus 
\[
\int_{Y^{\delta_1}\cap A}g\Big(\frac{d D u}{d |D u|} \Big)\,d |D u| 
=
\int_{Y^{\delta_1}\cap A}g\Big(\frac{d D u}{d |D u|} \Big)\, |\nabla u|\,dx + 
\int_{Y\cap A}g\Big(\frac{d D^s u}{d |D^s u|}\Big)\,d|D^s u|
\]
Since $f$ is quasiconvex, it and also $g$ is convex along rank-one directions 
\cite[pp.~299--300]{AmbFP00}; thus $g(\xi)=\lim_{t\to \infty}\frac{f(t\xi)}{t}=f^\infty(\xi)$ 
when $\xi$ is a rank-one matrix.
By Alberti's theorem, $\frac{d D^s u}{d |D^s u|}$ is a rank-one matrix \cite{A93}. 
Furthermore, $g(\xi)\le M$ when $|\xi|\le 1$. Hence we can continue our previous estimate 
by 
\begin{align*}
\int_{Y^{\delta_1}\cap A}g\Big(\frac{d D u}{d |D u|} \Big)\,d|D u| 
\le 
M\int_{Y^{\delta_1}\cap A} |\nabla u|\,dx 
+ \int_{Y\cap A}f^\infty\Big(\frac{d D^s u}{d |D^s u|} \Big)\,d|D^s u|.
\end{align*}

Combining the estimates from the previous paragraph with the earlier estimate over 
$A\setminus Y^{\delta_1}$ and assuming that $|D^s u|(\partial A)=0$ we obtain as $\delta_1\to 0$ that 
\[
\mathcal F(u,A)
\le 
C\int_A f(\nabla u)^{p(x)}+1\,dx + \int_{Y\cap A}f^\infty\Big(\frac{d D^s u}{d |D^s u|} \Big)\,d|D^s u|.
\]
In the same way as with \eqref{eq:acEst}, we conclude from this that 
\[
\mathcal F^s(u,B)\le \int_B f^\infty\Big(\frac{d D^s u}{d |D^s u|}\Big)\,d|D^s u|.
\qedhere
\]
\end{proof}

We are now ready to prove the main result. 

\begin{proof}[Proof of Theorem~\ref{thm:main}]
It follows from Lemma~\ref{lem:ubrelaxation} that 
\[
\mathcal F(u)\le \int_\Omega f(\nabla u)^{p(x)} \,dx +\int_\Omega f^\infty\Big(\frac{d D^s u}{d |D^s u|}\Big)\,d|D^s u|
\]
for all $u \in \BV^\px(\Omega; \R^m)$.

For the opposite inequality let $u\in \BV^\px(\Omega; \R^m)$ and $u_h \in W^{1,\px}(\Omega;\R^m)$ with 
$u_h \to u$ in $L^{\px}(\Omega;\R^m)$. By Proposition~\ref{prop:Fqc} and 
$u_h \in W^{1,\px}(\Omega;\R^m)$, 
\[
F_{qc}(u)
\le 
\liminf_{h\to \infty} F_{qc}(u_h)
=
\liminf_{h\to \infty} \int_\Omega f(\nabla u_h)^{p(x)}\, dx. 
\]
Taking the infimum over such sequences yields $F_{qc}(u)\le \mathcal F(u)$, as required. 
\end{proof}

The next result shows that the test-functions from $\mathcal F$ can alternatively be 
taken to be smooth. 
Note that uniform domains, including bounded domains with Lipschitz boundary, 
are extension domains \cite[Theorem~8.5.12]{DieHHR11}.

\begin{proposition}\label{lemmaC1} 
Let $\Omega \subset \mathbb R^n$ be a $W^{1,\px}$-extension domain, $p$ be $\log$-Hölder continuous and 
$f:\R^{n \times m} \to [0,\infty)$ satisfy \Hzero{}, \Htwo{}{} 
and \Hthree{}. 
Then 
\[
\mathcal F(u)
=
\inf\Big\{\liminf_{h\to \infty}\int_\Omega f(\nabla u_h)^{p(x)}\, dx \,\Big|\, 
u_h \in C^\infty(\overline \Omega;\R^m), u_h \to u \text{ in } L^{1}(\Omega;\R^m) \Big\},
\]
for every $u\in \BV^\px(\Omega;\R^m)$.
\end{proposition}
\begin{proof}[Proof]
Denote by $\mathcal F_{C^\infty}$ the right-hand side of the claimed equation. 
The  inequality  $\mathcal F(u) \le \mathcal F_{C^\infty}(u)$ follows from 
$C^\infty(\overline\Omega;\R^m) \subset W^{1,\px}(\Omega; \R^m)$. 

For the opposite inequality we follow the ideas from \cite[Proposition~2.6(ii)]{FonM92}. 
Let $u_j \in W^{1,\px}(\Omega;\R^m)$ with $u_j \to u \in L^\px(\Omega;\R^m)$ and 
\[
\mathcal F(u)=\lim_{j\to\infty} \int_\Omega f(\nabla u_j)^{p(x)}\, dx.
\]
By \cite[Theorem~9.1.7]{DieHHR11}, $C^\infty(\overline\Omega)$ is dense in $W^{1,p(\cdot)}(\Omega)$, 
hence, arguing componentwise one can deduce that $C^\infty(\overline{\Omega};\R^m)$ is dense in  
$W^{1,p(\cdot)}(\Omega;\R^m)$.
Thus we can find for every $u_j$ a sequence 
$(u_{j}^{k}) \subset C^\infty_0(\overline \Omega;\R^m)$ with 
$u_j^k \to u_j$ in $W^{1, \px}(\Omega;\R^m)$ 
and $\nabla u_j^k\to \nabla u_j$ almost everywhere. 
By the continuity of $f$ and $p$, we find that $f(\nabla u_j^k(x))^{p(x)} \to f(\nabla u_j(x))^{p(x)}$  
almost everywhere.
Assumption \Hthree{} and $p^+<\infty$ yield that $f(\xi)^{p(x)} \le (M (1+ |\xi|))^{p(x)} \le (2M)^{p^+}(1+ |\xi|^{p(x)})$ and hence 
$(2M)^{p^+}(1+|\nabla u_j^k|^\px) - f(\nabla u_j^k)^\px \ge 0$. Hence Fatou's lemma yields that
\[
\int_{\Omega} (2M)^{p^+}(1+|\nabla u_j|^{p(x)})- f(\nabla u_j)^{p(x)} \,\,dx \le \liminf_{k \to \infty}\int_{\Omega} (2M)^{p^+}(1+|\nabla u_j^k|^{p(x)})- f(\nabla u_j^k)^{p(x)} \,\,dx.
\]
Since $\Omega$ is bounded and $u_j^k \to u_j$ in $W^{1, \px}(\Omega, \R^m)$, 
the first terms cancel and we find that 
\[
\int_{\Omega} f(\nabla u_j)^{p(x)} \,\,dx \ge \limsup_{k \to \infty}\int_{\Omega} f(\nabla u_j^k)^{p(x)} \,\,dx.
\]
For every $j$ we choose $k_j$ such that  $\|u_j - u_j^{k_j}\|_1\le \frac{1}{j}$ and 
\[
\int_{\Omega} f(\nabla u_j^{k_j})^{p(x)} \,\,dx < \int_{\Omega} f(\nabla u_j)^{p(x)} \,\,dx
+ \frac{1}{j}.
\]
Then $(u_j^{k_j})$ converges to $u$ in $L^1(\Omega, \R^m)$
so it is a valid test sequence for $\mathcal F_{C^\infty}$. 
Hence 
\[
\begin{split}
\mathcal F_{C^\infty}(u) & \le
\liminf_{j \to \infty }\int_{\Omega} f(\nabla u_j^{k_j})^{p(x)} \,\,dx 
\le \liminf_{j \to \infty } \Big(\int_{\Omega} f(\nabla u_j)^{p(x)} + \frac1j \Big)
\le \mathcal{F} (u). \qedhere
\end{split}
\]
\end{proof}

{\bf Acknowledgements} The last author is a member of GNAMPA-INdAM whose support is gratefully acknowledged. She is also indebted with Gianni Dal Maso and Ana Margarida Ribeiro for helpful discussions on the subject of this paper.


%

\end{document}